\documentclass[review,compress]{elsarticle}
\usepackage{lineno,hyperref}
\usepackage{amsmath,amsthm,amssymb}
\usepackage{multirow}
\usepackage{rotating}
\usepackage{algorithm}
\usepackage{algcompatible}
\usepackage{graphicx} 
\usepackage{epstopdf}
\usepackage{color}
\usepackage{ulem}

\modulolinenumbers[5]

\journal{XXXX}
\newtheorem{theorem}{Theorem}[section]

\newtheorem{lemma}{Lemma}[section]

\newtheorem{remark}{Remark}[section]

\allowdisplaybreaks[4]

\begin{document}

\begin{frontmatter}
\title{Comments on lumping the Google matrix
\footnote{This work was supported by National Natural Science Foundation of China (Grant No. 12001363, 71671125, 12001395), Natural Science Foundation of Shanxi province, China (Grant No. 201901D211423).}}

\author[mymainaddress]{Yongxin Dong}
\author[mysecondaryaddress]{Yuehua Feng\corref{mycorrespondingauthor}}
\cortext[mycorrespondingauthor]{Corresponding author}
\ead{yhfeng@sues.edu.cn}
\author[mymainaddress]{Jianxin You}
\author[3]{Jinrui Guan}

\address[mymainaddress]
{\scriptsize School of Economics and Management,
Tongji University, Shanghai 200092, China
}
\address[mysecondaryaddress]
{\scriptsize School of Mathematics, Physics and Statistics,
Shanghai University of Engineering Science, Shanghai 201620, China}

\address[3]{\scriptsize Department of Mathematics, Taiyuan Normal University, Shanxi 030619, China}

\begin{abstract}
On the case that the number of dangling nodes is large, PageRank computation can be proceeded with a much smaller matrix through lumping all dangling nodes of a web graph into a single node. Thus, it saves many computational cost and operations. There are also some theoretical contributions on Jordan canonical form of the Google matrix. Motivated by these theoretical contributions, in this note, we provide alternative proofs for some results of Google matrix through the lumping method due to Ipsen and Selee. Specifically
we find that the result is also suitable for some subsequent work based on lumping dangling nodes into a node. Besides, an entirely new proof from the matrix decomposition viewpoint is also proposed.

\end{abstract}



\end{frontmatter}

\section{Preliminary}

Google's PageRank is a web link analysis model in the field of modern web information retrieval. Its importance lies in its ranking method. PageRank model
uses adjacency matrices and hyperlink matrices to describe web link structure graph. It is a ranking method that reveals the relative importance of corresponding web pages. The traditionary PageRank model is one of the first order Markov chain applications \cite{Langville2006google}. Based on Markov chain model which concerns a sequence of random variables, the state of a certain system at one time epoch only depends on the previous time epoch state \cite{meyer1989stochastic}. The Google matrix is a stochastic matrix with all its entries nonnegative and all its row sum equals 1. The nonnegative stochastic matrix has a largest eigenvalue 1 which is its spectral radius. For researches on stochastic matrices, we refer to readers to \cite{friedland1979lower,KIRKLAND2006laa,kirkland2009cycle}.
According to the survey in \cite{vigna2016spectral}, this kind of ranking methods belongs to the spectral rankings. In \cite{franceschet2011pagerank}, a short brief history of PageRank is summarized and contained in Table \ref{tab:HistoryPageRank}.
\begin{center}
\begin{table}[htbp]
\centering
\small
\caption{A short history of PageRank}
\label{tab:HistoryPageRank}
\begin{tabular}{llll}
  \hline
  &Year & Author & Contribution \\
\hline
  &1906 & Markov & Markov theory \\
  &1907 & Perron & Perron theorem \\
  &1912 & Frobenius & Perron-Frobenius theorem \\
  \multirow{2}{*}{} &1929 & von Mises and & Power method \\
 & &Pollaczek-Geiringer &\\
  &1941 & Leontief & Econometric mode \\
  &1949 & Seeley & Sociometric model \\
  &1952 & Wei & Sport ranking model \\
  &1953 & Katz & Sociometric model \\
  &1965 & Hubbell & Sociometric model \\
  &1976 & Pinski and Narin & Bibliometric model \\
  &1998 & Kleinberg & HITS \\
  &1998 & Brin and Page & PageRank \\
  \hline
\end{tabular}
\end{table}
\end{center}

As the web is indeed very huge, the web may contain over a billion pages. And a count of the web size increases quickly and dynamically. The stationary probability vector (or principle eigenvector) of a large Markov chain has widely applications. This is due to its nonnegativity and all entries sum 1. However, the principal eigenvector computation of a large Markov chain matrix faces great challenges since the involved computation matrix size can be even over a billion nodes (pages) \cite{Page1998StanfordDL,Brezinski2006simax}. It may take much time (several hours or days) to compute a large PageRank vector. Therefore, faster methods are needed to be studied. The simplest and oldest power methods require constant memory \cite{Berkhin2005IM}, and it converges slowly as the damping factor increases \cite{Langville2006google}. To improve the computational efficiency of PageRank, the lumping methods \cite{Ipsen2007simax,Lin2009computing,FengDong2021bims}, the Arnoldi-type algorithm \cite{Golub2006bit} and other theoretical and numerical results are available, see \cite{yumiaowuwei2012KAP,Dong2017Calcolo,MIAO2020112891,TIAN2021113295}. For the properties of the Google matrix, Horn and Serra-Capizzano and Serra-Capizzano proposed its analytic expression of the Jordan canonical form  \cite{horn2006general,Capizzano2005simax}. However, Wu and Wei further pointed out that the main theory due to Capizzano can be used to estimate the condition number of the PageRank vector as a function of the damping factor $\alpha$. They gave insightful researches on how to minimize the condition number of the PageRank computation problem by choosing a clever scaling matrix \cite{wu2008comments}. As it is mentioned above, the PageRank problem can be accelerated and saved computational cost by lumping all dangling nodes of a web graph into a single node \cite{Ipsen2007simax,lee2003fast}, if the dangling node number is large.

To make sure the uniqueness of stationary probability distribution of a particular matrix, Brin and Page invent a new matrix by taking a linear  combination of two stochastic matrices \cite{Page1998StanfordDL}
\begin{equation}\label{equ:google matrix}
G = \alpha S + (1-\alpha) E,~ S = H + d w^T,~ E = ev^T,~ G\in \mathbb{R}^{n\times n}
\end{equation}
where $w\geq 0,w\in \mathbb{R}^n,~\|w\|_1\equiv w^Te=1$
is the same dangling node vector,
$\alpha\in(0,1)$ is the damping factor,
$v\geq 0,v\in \mathbb{R}^n,\|v\|_1\equiv v^Te=1$ is a personalization vector,
$e$ is the vector of all ones with suitable size,
the entries of the dangling node indicator vector $d$ is defined by
\begin{equation}\label{equ:hijandd}
d_i = \left\{
\begin{array}{l}
1,~~\mbox{if page $i$ has no outlinks}, \\
 0,~~\rm{otherwise,} \\
 \end{array} \right.
\end{equation}
and the web link structure matrix is given by
\begin{equation}
h_{ij} = \left\{ \begin{array}{l}
 \frac{1}{n_i},~~ \mbox{a nonzero integer $n_i$ stands for the number of outlinks of page $i$ to page $j$}, \\ 
 0,~~\rm{otherwise}, \\
 \end{array} \right.
 \end{equation}
Not all web pages has links to other web pages, if web pages has no outlinks (i.e. pdf, image files, some web pages with on links to other pages), web pages are called dangling nodes; otherwise they are nondangling nodes. The PageRank vector $\pi$ is defined by computing the principle eigenvector of the Google matrix $G$
\begin{equation}\label{equ:pagerank problemROWCOL}
	\pi^T = \pi^T G, ~\mbox{or} ~\\
	\pi = G^T \pi,~\mbox{with}~ \pi\geq 0, \|\pi\|_1 = 1.
\end{equation}

The Google matrix $G$ is a primitive stochastic matrix \cite{meyer1989stochastic}, thus it is irreducible and aperiodic. This property can ensure the existence and uniqueness of the nonnegative dominant eigenvector due to the ranking matrix. For details, see Markov chains or Perron-Frobenius theorem \cite{Golub2006bit,horn1994topics}. By the lumping Google matrix method, Ipsen and Selee analyzed the relationship between rankings of nondangling nodes and rankings of dangling nodes \cite{Ipsen2007simax}. To further demonstrate the ranking relationship between nondangling nodes and dangling nodes during PageRank computation, we try to make theoretical contributions for clarifying their ranking relationship. Consequently, in this paper, from an aspect of theoretical analysis, we derive alternative proofs for some results on lumping the Google matrix.

\section{Lumping and related theorems}

In this section, we first review the lumping in matrix term which is given by Ipsen and Selee in \cite{Ipsen2007simax}. Let $P$ be a permutation matrix and
\begin{align*}
PMP^T=\begin{bmatrix}
M_{11}&\cdots&M_{1,k+1}\\
\vdots&\ddots&\vdots\\
M_{k+1,1}&\cdots&M_{k+1,k+1}
\end{bmatrix}
\end{align*}
be a partition of a stochastic matrix $M$. Then we call $M$ is lumpable with respect to this partition if each vector $M_{ij}e$ is a multiple of $e$ $(e=\begin{bmatrix}1 &\cdots &1\end{bmatrix}^T)$, $i\neq j$, $1\leq i,j \leq k+1$. The Google matrix is said to be lumpable if all dangling nodes are lumped into a single node \cite{Ipsen2007simax,lee2003fast}.

For ease of illustration, the related theorems proposed by Ipsen and Selee \cite{Ipsen2007simax} are reviewed in the following discussions.
Suppose that $G \in \mathbb{R}^{n\times n}$ has $k$ nondangling nodes and $n-k$ dangling nodes, thus there exists a permutation matrix $\Pi\in \mathbb{R}^{n\times n}$, such that
\begin{align*}
\widetilde{G} = \Pi G \Pi^T.
\end{align*}
To find the permutation matrix $\Pi$, we perform the following steps.
For a row stochastic web link matrix $H$, we first compute $h=He$, where $e=\begin{bmatrix}1&1&\cdots&1\end{bmatrix}^T$. Then in exact arithmetic, $h$ is a nonnegative vector with entries $1$ or $0$. Hence, if we define a dangling node set $\text{D}=\{i|h_i=0 \}$, and a nondangling node set $\text{ND}=\{i |h_i=1 \}$. A suitable size identity matrix is denoted by $I$. The proper permutation matrix $\Pi$ can be defined by $\Pi=I([\text{ND},\text{D}],:)$ (in Matlab notation). Therefore, in numerical simulations, we first collect the index $i$ which satisfies $h_i=0$, and make a notation $\text{D}=\{i |h_i=0 \}$. The other indices ($\{1,2,\cdots,n\}\backslash\text{D}$) belongs to the set ND. This manipulation can avoid some practical numerical problems, such as the rounding off error. This is due to the fact that an element of $d$ may approximate 1, be not exactly 1, say, $0.999$, during double precision computation.
From \eqref{equ:google matrix}, we have
\begin{align*}
\widetilde{G}=&\alpha \Pi \left( H  + d w^T\right)\Pi^T + \left(1-\alpha\right)\Pi ev^T\Pi^T=
\alpha\left( \widetilde{H}+\widetilde{d} w^T\right)+ \left(1-\alpha\right)ev^T
\\
=&\alpha \widetilde{S}+(1-\alpha)ev^T,
\end{align*}
where $\widetilde{H} = \begin{bmatrix} \widetilde{H}_{11} & \widetilde{H}_{12} \\ 0 & 0 \end{bmatrix},\quad \widetilde{d} = \begin{bmatrix} 0\\ e \end{bmatrix}\in \mathbb{R}^{n},\quad w = \begin{bmatrix} w_1\\ w_2\end{bmatrix}\in \mathbb{R}^{n},\quad v = \begin{bmatrix} v_1\\ v_2\end{bmatrix}\in \mathbb{R}^{n},\quad u= \begin{bmatrix} u_1\\ u_2\end{bmatrix}\in \mathbb{R}^{n}
$, $\widetilde{H}_{11} \in \mathbb{R}^{k\times k}$,
$\widetilde{H}_{12}\in \mathbb{R}^{k\times (n-k)}$, $u_1 \in \mathbb{R}^{ k }$,
$u_2 \in \mathbb{R}^{ n-k  }$, $u=\alpha w+\left(1-\alpha\right)v$, and $e$ is a column vector of all ones with suitable size.
In this way, the indices of pages are reordered. So it is clear that 
\begin{align}\label{equ:G22block}
\widetilde{G}
=&
\alpha
\begin{bmatrix}
\widetilde{H}_{11} & \widetilde{H}_{12} \\
e w_1^T & e w_2^T \end{bmatrix}+
\left(1-\alpha\right)
\begin{bmatrix} ev_1^T & ev_2^T \\
ev_1^T & ev_2^T \end{bmatrix}\nonumber \\
=&
\begin{bmatrix}
\alpha\widetilde{H}_{11}+(1-\alpha)ev_1^T
&\alpha \widetilde{H}_{12}+(1-\alpha)ev_2^T \\
e (\alpha w_1^T+(1-\alpha)v_1) & e(\alpha w_2^T+(1-\alpha)v_2 )
\end{bmatrix}\nonumber \\
=&
\begin{bmatrix}\widetilde{G}_{11} & \widetilde{G}_{12} \\
e u_1^T & e u_2^T \end{bmatrix}.
\end{align}
Moreover, if
\begin{align}
\widetilde{\pi}^T =  \widetilde{\pi}^T \widetilde{G} = \widetilde{\pi}^T \Pi G \Pi^T,\quad \mbox{where $\widetilde{\pi}\geq 0, \|\widetilde{\pi}\|_1 = 1$},
\end{align}
then the PageRank vector corresponding to \eqref{equ:pagerank problemROWCOL} is
\begin{equation*}
	\pi^T = \widetilde{\pi}^T \Pi.
\end{equation*}

By using the lumping method (or similarity transformation) and the similarity transformation matrix
\begin{align*}
L=I_{n-k}-\frac{1}{n-k}\widehat{e}e^T,\quad\mbox{with}\quad \widehat{e}=e-e_1=\begin{bmatrix}0,1,\cdots,1\end{bmatrix}\in \mathbb{R}^{n-k},
\end{align*}
where $I_{n-k}=\begin{bmatrix}e_1\cdots e_{n-k}\end{bmatrix}$ denotes the identity matrix of order $n-k$, and $e_i(i=1,2,\cdots,n-k)$ is its $i-$th column vector. Ipsen and Selee obtained the following theorems in \cite{Ipsen2007simax}.
They showed the relationship between the PageRank $\widetilde{\pi}$ of $\widetilde{G}\in \mathbb{R}^{n\times n}$
and the stationary distribution $\sigma$ of $\widetilde{G}^{(1)}\in \mathbb{R}^{(k+1)\times (k+1)}$.
For details of the analytic induction, see \cite{Ipsen2007simax}.

\begin{theorem}
\cite{Ipsen2007simax}\label{thm:ipsen1}
With the above notation, let
\begin{equation}
X=\begin{bmatrix}
I_k& 0\\
0& L
\end{bmatrix},\quad \mbox{where} \quad L = I_{n-k} - \frac{1}{n-k}\widehat{e}e^T \quad \mbox{and} \quad
\widehat{e} = e - e_1= \begin{bmatrix} 0 &1 &\cdots &1 \end{bmatrix}^T.
\end{equation}
Then
$X\widetilde{G}X^{-1} =
\begin{bmatrix} \widetilde{G}^{(1)} & \widetilde{G}^{(2)} \\
0 & 0\end{bmatrix},$ where
\begin{equation}\label{equ:G1G2}
 \widetilde{G}^{(1)}  =
\begin{bmatrix}\widetilde{G}_{11} & \widetilde{G}_{12}e \\
u_1^T & u_2^T e \end{bmatrix}\quad \mbox{and} \quad \widetilde{G}^{(2)}=\begin{bmatrix} \widetilde{G}_{12} \left(I+\widehat{e}e^T \right)\\ u_2^T\left(I+\widehat{e}e^T \right) \end{bmatrix}
\begin{bmatrix}e_2 &e_3 &\cdots &e_{n-k} \end{bmatrix}.
\end{equation}
The stochastic matrix $\widetilde{G}^{(1)}$ of order $k+1$ has the same nonzero eigenvalues as $\widetilde{G}$.
\end{theorem}

\begin{theorem}
\cite{Ipsen2007simax}\label{thm:ipsen2}
With the above notation, let
\begin{align}\label{eqn:lumpedHatG(1)}
\sigma^T \begin{bmatrix}\widetilde{G}_{11} & \widetilde{G}_{12}e \\ u_1^T & u_2^T e \end{bmatrix} =
\sigma^T, \sigma \geq 0, \|\sigma\|_1 = 1,
\end{align}
and partition $\sigma^T = \begin{bmatrix} \sigma^T_{1:k} & \sigma_{k+1}\end{bmatrix}$,
where $\sigma_{k+1}$ is a scalar.
Then the PageRank vector of $\widetilde{G}$ equals
\begin{equation}\label{eqn:piT}
\widetilde{\pi}^T =
\begin{bmatrix} \sigma^T_{1:k} & \sigma^T
\begin{bmatrix}\widetilde{G}_{12} \\ u_2^T\end{bmatrix}\end{bmatrix},\quad \mbox{where} \quad
\widetilde{G}_{12} =
\alpha \widetilde{H}_{12}+(1-\alpha)e v_2^T.
\end{equation}

\end{theorem}
As it was stated in \cite{Ipsen2007simax}, the leading $k$ elements of $\widetilde{\pi}$ represent the PageRank associated with the nondangling nodes, and the trailing $n-k$ elements stand for the PageRank of the dangling nodes. Motivated by validating the analytic relationship between them, in this paper, we derive alternative proofs for lumping the PageRank problem.

\section{Proposed transformation matrices}

In the section, we develop an alternative proof in \cite{Ipsen2007simax}
based on a new similarity transformation matrix. They employ the matrix \begin{align}\label{equ:X}
X=
\begin{bmatrix}
I_k&0\\
0&L
\end{bmatrix},
L=I_{n-k}-\frac{1}{n-k}\widehat{e}e^T,\widehat{e} = e - e_1= \begin{bmatrix} 0& 1 &\cdots& 1\end{bmatrix}^T,
\end{align}
where $e$ is a vector of all ones and $e_1$ is the first column of a suitable size identity matrix. In Theorem \ref{thm:ipsen1}, if we employ the following matrix
\begin{align}\label{equ:widetildeX}
\widetilde{X}=
\begin{bmatrix}
I_k&0\\
0&\widetilde{L}
\end{bmatrix},
\widetilde{L}=I_{n-k}-\widehat{e}e_1^T,\widehat{e} = e - e_1= \begin{bmatrix} 0& 1 &\cdots& 1\end{bmatrix}^T,
\end{align}
then we will show you that the related theorems also holds. For comparing \eqref{equ:X} and \eqref{equ:widetildeX}, we separate the leading 1 row and column of $\widetilde{L}$
\begin{align}\label{eq:tildeL}
\widetilde{L}=\begin{bmatrix}
1&0\\
-e&I_{n-k-1}
\end{bmatrix},
\end{align}
and partition $L$ conformally with $\widetilde{L}$,
\begin{align}
L=\begin{bmatrix}
1&0\\
-\frac{1}{n-k}e&I_{n-k-1}-\frac{1}{n-k}ee^T
\end{bmatrix}.
\end{align}
Specifically, the new $(n-k)\times(n-k)$ matrix $$\widetilde{L}=
\begin{bmatrix}
1 &0&0&\cdots&0 \\
-1&1&0&\cdots&0\\
-1&0&1&\cdots&0\\
\vdots&\vdots&\ddots&\ddots&\vdots\\
-1&0&\cdots&0&1\end{bmatrix}
$$
is triangular and relatively sparse, while the $(n-k)\times (n-k)$ matrix 
$$L=\begin{bmatrix}
1 &0&0&\cdots&0 \\
-\frac{1}{n-k}&\frac{n-k-1}{n-k}&-\frac{1}{n-k}&\cdots  &-\frac{1}{n-k}\\
-\frac{1}{n-k}&-\frac{1}{n-k}&\frac{n-k-1}{n-k}&\cdots  &-\frac{1}{n-k}\\
\vdots&\vdots&\ddots&\ddots&\vdots\\
-\frac{1}{n-k}&-\frac{1}{n-k}&-\frac{1}{n-k}&\cdots&\frac{n-k-1}{n-k}\end{bmatrix}$$
is a relative dense matrix. The 2-by-2 block of $\widetilde{L}$ is the identity matrix, while the 2-by-2 block of $L$ is a typical dense matrix. Inspired and motivated by this property, we consider replacing $L$ with $\widetilde{L}$ in $X$ so that an alternative proof process can be presented. Thus, the corresponding proof process can be simplified and shortened. Moreover, we also find that if we take $$\widetilde{L}=I_{n-k}-J_{n-k}(0,n-k)=\begin{bmatrix}1 &0&0&\cdots&0 \\
-1&1&0&\cdots&0\\
0&-1&1&\cdots&0\\
\vdots&\vdots&\ddots&\ddots&\vdots\\
0&0&\cdots&-1&1\end{bmatrix},$$
in \eqref{equ:widetildeX}, where $J_{n-k}(0,n-k)$ is a lower Jordan block of order $n-k$ with diagonal entries $0$, the theorems and the process of the above theorems are also holds. That is to say, the similarity transformation matrix exists but is not unique at all. As a result, in this paper, we generalize the condition of the above theorems by generalizing the similarity transformation matrix condition. Instead of the matrix $L$ which is defined in \cite{Ipsen2007simax}, we propose a class of invertible matrices which satisfy the following condition
\begin{align}\label{eqn:L-condition}
\widehat{L}e=e_1,\mbox{where $\widehat{L}$ is an invertible matrix of order $n-k$}.
\end{align}
Note that $\widehat{L}$ contains specific matrices $\widetilde{L}$ or $L$.  
If $\widehat{L}=L$, then the condition \eqref{eqn:L-condition} becomes the condition for the original theorem and proof process due to Ipsen and Selee. Before our alternative proofs, we propose the following Lemma.

\begin{lemma}\label{lem:widetildeG}
With the above notation, let
\begin{align*}
\widetilde{G}
=
\begin{bmatrix}\widetilde{G}_{11} & \widetilde{G}_{12} \\
e u_1^T & e u_2^T \end{bmatrix},
\end{align*}
which is defined in \eqref{equ:G22block}, and
$\widetilde{X}=
\begin{bmatrix}
I_k&0\\
0&\widetilde{L}
\end{bmatrix},$ where $\widetilde{L}=I_{n-k}-\widehat{e}e_1^T$
then
\begin{align}\label{equ:widetildeXGX-1}
\widetilde{X}\widetilde{G}\widetilde{X}^{-1}=
\begin{bmatrix}
\widetilde{G}_{11} &\widetilde{G}_{12}\widetilde{L}^{-1}\\
\widetilde{L}eu_1^T&\widetilde{L}eu_2^T\widetilde{L}^{-1}
\end{bmatrix}
=\begin{bmatrix}
\widetilde{G}_{11}& \widetilde{G}_{12}e&\widetilde{G}_{12}(I+\hat{e}e_1^T)\begin{bmatrix}e_2& \cdots& e_{n-k}\end{bmatrix}\\
u_1^T& u_2^Te &u_2^T(I+\hat{e}e_1^T)\begin{bmatrix}e_2 &\cdots&e_{n-k}\end{bmatrix}\\
0&0&0
\end{bmatrix}.
\end{align}

\end{lemma}
\begin{proof}
The proof is straightforward by computing expression. By direct computation, we have
\begin{align*}
\widetilde{X}\widetilde{G}\widetilde{X}^{-1}=&
\begin{bmatrix}
\widetilde{G}_{11} &\widetilde{G}_{12}\widetilde{L}^{-1}\\
\widetilde{L}eu_1^T&\widetilde{L}eu_2^T\widetilde{L}^{-1}
\end{bmatrix} \\
=&\begin{bmatrix}
\widetilde{G}_{11}&\widetilde{G}_{12}\widetilde{L}^{-1}e_1
&\widetilde{G}_{12}\widetilde{L}^{-1}\begin{bmatrix}e_2&\cdots&e_{n-k}\end{bmatrix}\\
e_1^T\widetilde{L}eu_1^T &e_1^T\widetilde{L}eu_2^T\widetilde{L}^{-1}e_1
&e_1^T\widetilde{L}eu_2^T\widetilde{L}^{-1}
\begin{bmatrix}e_2&\cdots&e_{n-k}\end{bmatrix}\\
\begin{bmatrix}e_2^T\\ \vdots\\ e_{n-k}^T\end{bmatrix}\widetilde{L}eu_1^T
&\begin{bmatrix}e_2^T\\ \vdots\\ e_{n-k}^T\end{bmatrix}
\widetilde{L}eu_2^T\widetilde{L}^{-1}e_1
&\begin{bmatrix}e_2^T\\ \vdots\\ e_{n-k}^T\end{bmatrix}
\widetilde{L}eu_2^T\widetilde{L}^{-1}\begin{bmatrix}e_2
&\cdots&e_{n-k}\end{bmatrix}
\end{bmatrix}.
\end{align*}
Using the fact that $\widetilde{L}e=e_1$ and $\widetilde{L}^{-1}e_1=\left(I+\widehat{e}e_1^T\right)e=e$, we derive \eqref{equ:widetildeXGX-1}.
\end{proof}

We further confirm that if an invertible matrix $\widehat{L}$ satisfies $\widehat{L}e=e_1$, then the above conclusion also holds. Hence, we have the following Lemma.
\begin{lemma}\label{lem:widetildeG-gen}
With the above notation,
let \begin{align*}
\widetilde{G}
=
\begin{bmatrix}\widetilde{G}_{11} & \widetilde{G}_{12} \\
e u_1^T & e u_2^T \end{bmatrix},
\end{align*}
which is defined in \eqref{equ:G22block}, and
$\widehat{X}=
\begin{bmatrix}
I_k&0\\
0&\widehat{L}
\end{bmatrix},$ where the invertible matrix $\widehat{L}$ satisfies $\widehat{L}e=e_1$, then
\begin{align}\label{equ:widetildeXGX-1-2}
\widehat{X}\widetilde{G}\widehat{X}^{-1}=&
\begin{bmatrix}
\widetilde{G}_{11} &\widetilde{G}_{12}\widehat{L}^{-1}\\
\widehat{L}eu_1^T&\widetilde{L}eu_2^T\widehat{L}^{-1}
\end{bmatrix}\nonumber \\
=&\begin{bmatrix}
\widetilde{G}_{11}& \widetilde{G}_{12}e&\widetilde{G}_{12}\widehat{L}^{-1}\begin{bmatrix}e_2& \cdots& e_{n-k}\end{bmatrix}\\
u_1^T& u_2^Te &u_2^T\widehat{L}^{-1}\begin{bmatrix}e_2 &\cdots&e_{n-k}\end{bmatrix}\\
0&0&0
\end{bmatrix}
\end{align}

\end{lemma}

The proof process of this Lemma is similar to that of Lemma \ref{lem:widetildeG}, so we omit it here.

\section{Alternative Proofs of Theorems \ref{thm:ipsen1} and \ref{thm:ipsen2}}
\subsection{An alternative proof of Theorem \ref{thm:ipsen1}}
Now we are ready to present an alternative proof of Theorem \ref{thm:ipsen1} below.

\begin{proof}

By separating the leading $(k+1)\times (k+1)$ submatrix, according to Lemma \ref{lem:widetildeG-gen}, we rewrite
\begin{equation}\label{equ:widetilde_X widetilde_G widetilde_X}
\widehat{X}\widetilde{G}\widehat{X}^{-1} =
\begin{bmatrix} \widetilde{G}^{(1)} & \widetilde{G}^{(2)} \\
0 &0\end{bmatrix},
\end{equation}
where
\begin{equation*}
\widetilde{G}^{(1)}  =
\begin{bmatrix}\widetilde{G}_{11} & \widetilde{G}_{12}e \\
u_1^T & u_2^T e \end{bmatrix},\quad \mbox{and}\quad
\widetilde{G}^{(2)}=
\begin{bmatrix}
\widetilde{G}_{12} \\
u_2^T \end{bmatrix}\widehat{L}^{-1}
\begin{bmatrix}e_2&\cdots&e_{n-k} \end{bmatrix}.
\end{equation*}

\end{proof}
We remark that the proof process can be achieved if the invertible similarity transformation  matrix $\widehat{L}$ satisfies the condition $\widehat{L}e=e_1$.

\subsection{An alternative proof of Theorem \ref{thm:ipsen2}}
In this subsection, instead of a more complicated and denser matrix $L$ in Theorem \ref{thm:ipsen1}, we choose a general invertible matrix $\widehat{L}$ satisfying $\widehat{L}e=e_1$ to present an alternative proof of Theorem \ref{thm:ipsen2}. 
Now our proof are shown below.

\begin{proof}
According to theorem \ref{thm:ipsen1}, the stochastic matrix $\widetilde{G}^{(1)}$ of order $k+1$ has the same nonzero eigenvalues as $\widetilde{G}$. From \eqref{eqn:lumpedHatG(1)} and \eqref{equ:widetilde_X widetilde_G widetilde_X}, we can obtain that $\left[\sigma^T\right.  \left. \sigma^T \widetilde{G}^{(2)}  \right]$ is an eigenvector for $\widehat{X}\widetilde{G} \widehat{X}^{-1}$ associated with the eigenvalue $\lambda=1$. Therefore,
\begin{align}
\widehat{\pi}^T=\begin{bmatrix}\sigma^T&\sigma^T \widetilde{G}^{(2)} \end{bmatrix}\widehat{X}
\end{align}
is an eigenvector of $\widetilde{G}$ associated with $\lambda=1$. Since $\widetilde{G}$ and $\widetilde{G}^{(1)}$ have the same nonzero eigenvalues, and the principle eigenvalue of $\widetilde{G}$ is distinct, the stationary probability distribution $\sigma$ of $\widetilde{G}^{(1)}$ is unique. We repartition
\begin{align}
\widehat{\pi}^T=
\begin{bmatrix}
\sigma_{1:k}^T
&\begin{bmatrix}\sigma_{k+1}  &  \sigma^T \widetilde{G}^{(2)} \end{bmatrix}
\end{bmatrix}
\begin{bmatrix}I_k&0\\
0& \widehat{L}
\end{bmatrix}.
\end{align}
Multiplying out
\begin{align}
\widehat{\pi}^T=\begin{bmatrix}\sigma_{1:k}^T
&\begin{bmatrix}\sigma_{k+1} & \sigma^T \widetilde{G}^{(2)} \end{bmatrix}
\widehat{L} \end{bmatrix}.
\end{align}

Partitioning $\widehat{L} = \begin{bmatrix}
\widehat{L}_{11}&\widehat{L}_{12}\\
\widehat{L}_{21}&\widehat{L}_{22}
\end{bmatrix}$ with a scalar $\widehat{L}_{11}$, a row vector $\widehat{L}_{12}$, a column vector $\widehat{L}_{21}$, and a square submatrix $\widehat{L}_{22}$, hence,
\begin{align}
&\begin{bmatrix}
\sigma_{k+1} & \sigma^T \widetilde{G}^{(2)}
\end{bmatrix}
\begin{bmatrix}
\widehat{L}_{11}&\widehat{L}_{12}\\
\widehat{L}_{21}&\widehat{L}_{22}
\end{bmatrix}
=
\begin{bmatrix}
\sigma_{k+1}\widehat{L}_{11} + \sigma^T \widetilde{G}^{(2)}\widehat{L}_{21} &\sigma_{k+1}\widehat{L}_{12}+\sigma^T\widetilde{G}^{(2)}\widehat{L}_{22}
\end{bmatrix}\nonumber
\\=&
\begin{bmatrix}
\sigma^T\begin{bmatrix}\widetilde{G}_{12}\\u_2^T
\end{bmatrix}\widehat{L}^{-1}e_1\widehat{L}_{11}+
\sigma^T
\begin{bmatrix}\widetilde{G}_{12} \\
u_2^T \end{bmatrix}\widehat{L}^{-1}
\begin{bmatrix}e_2 &\cdots &e_{n-k}\end{bmatrix}\widehat{L}_{21} &\sigma_{k+1}\widehat{L}_{12}+\sigma^T\widetilde{G}^{(2)}\widehat{L}_{22}
\end{bmatrix}\nonumber
\\=&
\begin{bmatrix}
\sigma^T\begin{bmatrix}\widetilde{G}_{12}\\u_2^T
\end{bmatrix}\widehat{L}^{-1}e_1\widehat{L}_{11}+\sigma^T
\begin{bmatrix}\widetilde{G}_{12} \\
u_2^T \end{bmatrix}\widehat{L}^{-1}
\begin{bmatrix}e_2 &\cdots &e_{n-k}\end{bmatrix} \widehat{L}_{21} &
\sigma^T\begin{bmatrix}\widetilde{G}_{12}\\u_2^T
\end{bmatrix}\widehat{L}^{-1}e_1\widehat{L}_{12}+\sigma^T \widetilde{G}^{(2)}\widehat{L}_{22}
\end{bmatrix}
\nonumber
\\=&
\begin{bmatrix}
\sigma^T\begin{bmatrix}\widetilde{G}_{12}\\u_2^T
\end{bmatrix}\widehat{L}^{-1}\begin{bmatrix}\widehat{L}_{11}\\\widehat{L}_{21}
\end{bmatrix}
&\sigma^T\begin{bmatrix}\widetilde{G}_{12}\\u_2^T
\end{bmatrix}\widehat{L}^{-1}\begin{bmatrix}\widehat{L}_{12}\\ \widehat{L}_{22}\end{bmatrix}
\end{bmatrix}\nonumber
\\=&\sigma^T\begin{bmatrix}\widetilde{G}_{12}\\u_2^T\end{bmatrix}
\end{align}
due to the fact that
\begin{align}
e=\widehat{L}^{-1}e_1, \quad
\sigma_{k+1}=\sigma^T
\begin{bmatrix}
\widetilde{G}_{12}e\\u_2^Te\end{bmatrix},\quad\mbox{and}\quad
\widetilde{G}^{(2)}=
\begin{bmatrix}
\widetilde{G}_{12} \\
u_2^T \end{bmatrix}\widehat{L}^{-1}
\begin{bmatrix}e_2&\cdots&e_{n-k} \end{bmatrix}.
\end{align}
Hence,
\begin{equation}
\widehat{\pi}^T =
\begin{bmatrix} \sigma^T_{1:k} & \sigma^T
\begin{bmatrix}\widetilde{G}_{12} \\ u_2^T\end{bmatrix}\end{bmatrix}.
\end{equation}
As discussed above and $\widetilde{\pi}$ is unique, we conclude that $\widehat{\pi}=\widetilde{\pi}$ if $e^T\widehat{\pi}=1$.
\end{proof}

\begin{remark}
A class of similarity transformation matrices $\widetilde{X}$ can provide an alternative proof of theorems \ref{thm:ipsen1} and \ref{thm:ipsen2}, and the proof process can be simpler and easier. Besides, Ipsen and Selee in \cite{Ipsen2007simax} extended the single class of dangling nodes to $m\geq 1$ different classes (Section 3.4 in \cite{Ipsen2007simax}), we stress that the concrete invertible matrices $L_1$ and $L_2$ constructed in \cite{Ipsen2007simax} can be also generalized. That is, if invertible matrices $\widehat{L}_1$ and $\widehat{L}_2$ satisfy the conditions $\widehat{L}_1e=e_1$ and $\widehat{L}_2e=e_1$, the corresponding conclusion and proof process still hold.
\end{remark}

\begin{remark}
In \cite{Lin2009computing}, the nondangling nodes are further classified into strongly nondangling nodes and weakly nondangling nodes, and thus the Google matrix has an $3\times 3$ block structure. The transformation matrix $L=I_{n-k}-\frac{1}{n-k}\widehat{e}e^T$, where $\widehat{e}=e-e_1=\begin{bmatrix}0&1&\cdots&1\end{bmatrix}$, which is used during the proof process of Theorem 2.1 in \cite{Lin2009computing}, can also be replaced by an invertible matrix $\widehat{L}$ satisfying the condition in \eqref{eqn:L-condition}. Meanwhile, after adopting the above new class of matrix, the main results in \cite{Lin2009computing} also hold.
\end{remark}

\begin{remark}
A minimal irreducible adjustment of PageRank was proposed in \cite{SongLi2013jamc}, they used an effective blocking and lumping algorithm for speeding up the PageRank computation. The new class of similarity transformation matrix $\widehat{L}$ can also play an important role during the proof process of Theorems 2 and 3 in \cite{SongLi2013jamc}. They not only derive the same theoretical results, but also shorten and simplify the process of theoretical analysis. 
\end{remark}

In the next subsection, some necessary preliminaries are introduced. Then from the matrix decomposition form, we validate Theorem \ref{thm:ipsen2}.

\subsection{Proof of Theorem \ref{thm:ipsen2} from matrix decomposition viewpoint}

From \cite{Berman1994siam,Ipsen2006simax}, we know that $I-\widetilde{G}$ is an M-matrix, as well as irreducible and singular. Hence, its nontrivial leading principle submatrix $I-\widetilde{G}_{11}$ is nonsingular \cite{Berman1994siam,Ipsen2006simax}. Therefore, Ipsen and Kirkland in \cite{Ipsen2006simax} expressed $\pi$ in terms of the block LDU decomposition
\begin{align}\label{eq:LDU}
I-\widetilde{G}&=
\begin{bmatrix}
I&0\\
-\widetilde{G}_{21}\left(I-\widetilde{G}_{11}\right)^{-1}&I
\end{bmatrix}
\begin{bmatrix}
I-\widetilde{G}_{11}&0\\
0&I-\widehat{S}
\end{bmatrix}
\begin{bmatrix}
I&-\left(I-\widetilde{G}_{11}\right)^{-1}\widetilde{G}_{12}\\
0&I
\end{bmatrix}
\end{align}
where $\widehat{S}=\widetilde{G}_{22}+\widetilde{G}_{21}(I-\widetilde{G}_{11})^{-1}
\widetilde{G}_{12}$. We know that $I-\widehat{S}$ is the Schur complement of $I-\widetilde{G}_{11}$ in $I-\widetilde{G}$. The matrix $\widehat{S}$ is also known as the stochastic complement of $\widetilde{G}_{22}$ in $P$. It is a special case of a Perron complement in the context of nonnegative matrices \cite{meyer1989stochastic,Langville2006google}. From \eqref{eq:LDU},  $\widetilde{\pi}^T(I-\widetilde{G})=0 $ if and only if $$\widetilde{\pi}^T\begin{bmatrix}
I&0\\
-\widetilde{G}_{21}\left(I-\widetilde{G}_{11}\right)^{-1}&I
\end{bmatrix}
\begin{bmatrix}
I-\widetilde{G}_{11}&0\\
0&I-\widehat{S}
\end{bmatrix}=0.$$ As $U$ is nonsingular, hence,
\begin{align}\label{eq:pi2pi1}
\widetilde{\pi}_2^T\widehat{S}=\widetilde{\pi}_2^T,\quad \mbox{and}\quad \widetilde{\pi}_1^T=\widetilde{\pi}_2^T\widetilde{G}_{21}
\left(I-\widetilde{G}_{11}\right)^{-1},
\end{align}
where $\widetilde{\pi}_2$ is a stationary distribution for the smaller matrix $\widehat{S}$. Keep in mind that $$\widehat{S}=\widetilde{G}_{22}+\widetilde{G}_{21}
\left(I-\widetilde{G}_{11}\right)^{-1}\widetilde{G}_{12},$$ then \eqref{eq:pi2pi1} can be reformulated as
\begin{align}\label{eq:pi2pi1couple}
\widetilde{\pi}_2^T=\widetilde{\pi}_1^T\widetilde{G}_{12}
\left(I-\widetilde{G}_{22}\right)^{-1},\quad \mbox{and}\quad \widetilde{\pi}_1^T=\widetilde{\pi}_2^T\widetilde{G}_{21}
\left(I-\widetilde{G}_{11}\right)^{-1}.
\end{align}

By comparing the differences between  $\widetilde{\pi}^T\begin{bmatrix}\widetilde{G}_{11} &\widetilde{G}_{12}\\
eu_1^T&eu_2^T\end{bmatrix}=\widetilde{\pi}^T$, with $\widetilde{\pi}\geq 0$, $\widetilde{\pi}^Te=1$ and
$\sigma^T\begin{bmatrix}\widetilde{G}_{11} &\widetilde{G}_{12}e\\
u_1^T&u_2^Te\end{bmatrix}=\sigma^T$, with $\sigma\geq 0$, $\sigma^Te=1$,
we begin our new proof based on above discussions and notations,

\begin{proof}

As \begin{align*}
\widetilde{X}\widetilde{G}\widetilde{X}^{-1}=
\begin{bmatrix}G^{(1)} &G^{(2)}\\
0&0 \end{bmatrix},\quad\mbox{where} \quad
\widehat{X}=\begin{bmatrix}
I_k&0\\
0&\widehat{L}
\end{bmatrix},
\widehat{L}e=e_1, \mbox{$\widehat{L}$ is an invertible matrix}\quad
\end{align*}
then $\widehat{\pi}= \begin{bmatrix}\sigma^T &\sigma^T\widetilde{G}^{(2)}\end{bmatrix}$
is an eigenvalue of $\widetilde{G}$ associated with eigenvalue $\lambda_1=1$ and is a multiple of probability distribution of $\widetilde{G}$. Since $\widetilde{G}$ and $\widetilde{G}^{(1)}$ have the same nonzero spectral, therefore $\widetilde{G}^{(1)}$ has a unique dominant eigenvector by considering the properties of $\widetilde{G}$. If $$e^T\widehat{\pi}=e^T\begin{bmatrix}\sigma^T &\sigma^T\widetilde{G}^{(2)}\end{bmatrix}=1,$$ we can conclude that the PageRank vector $\widetilde{\pi}$ of $\widetilde{G}$ equals $\widehat{\pi}$. Based on the assumption that $e^T\widehat{\pi}=1$, if we let the leading nondangling node ranking $\widetilde{\pi}_1^T= \widehat{\pi}_{1:k}= \sigma_{1:k}^T$, then from \eqref{eq:pi2pi1couple}, we obtain that
\begin{align*}
\widetilde{\pi}_2^T=\widetilde{\pi}_1^T
\widetilde{G}_{12}(I-\widetilde{G}_{22})^{-1} \quad \mbox{and}\quad
\widetilde{\pi}_1^T=
\widetilde{\pi}_2^T\widetilde{G}_{21}(I-\widetilde{G}_{11})^{-1}.
\end{align*}
Thus the tailing $n-k$ elements of $\widetilde{\pi}$ satisfies  $\widetilde{\pi}_2^T(I-\widetilde{G}_{22})=\sigma_{1:k}^T\widetilde{G}_{12}$, therefore,
$$\widetilde{\pi}_2^T-\widetilde{\pi}_2^T eu_2^T= \sigma_{1:k}^T\widetilde{G}_{12}.$$
Finally, we get the dangling node ranking
$$\widetilde{\pi}_2^T=\sigma_{k+1}u_2^T +\sigma_{1:k}^T\widetilde{G}_{12}=\sigma \begin{bmatrix}\widetilde{G}_{12}\\ u_2^T\end{bmatrix}. $$
Hence, we complete the proof due to the assumption that $\widehat{\pi}^Te=1$.

\end{proof}

\section{Conclusion}

In this paper, after a brief introduction of the spectral ranking history and the PageRank model, we have presented a class of similarity transformation matrices which are used in the proof of lumping PageRank computation problems in \cite{Ipsen2007simax}. This class of new matrices are consisted of structure  like $\begin{bmatrix}I&0\\0&\widehat{L}\end{bmatrix}$, where $I$ denotes a leading $k\times k$ identity matrix and $\widetilde{L}$ is an invertible matrix of order $n-k$ satisfying $\widehat{L}e=e_1$, where $e\in \mathbb{R}^{n-k}$ is a vector of all ones and $e_1$ is a canonical coordinate vector (i.e., the first column of the identity matrix $I_{n-k}$). As a result, the proof process due to Ipsen and Selee is simplified and shortened.

In addition, we also provide another proof of theorem \ref{thm:ipsen2} from the matrix decomposition viewpoint. The theorem \ref{thm:ipsen2} shows us the relationship between rankings of nondangling nodes and dangling nodes. That is, the rankings of nondangling nodes can be computed independently from that of dangling nodes; while rankings of dangling nodes depends on the ranking results of nondangling nodes. Finally, 
the spectral properties of a stochastic have an important impact on the power method's convergence. As a result, further researches may include how to accelerate the PageRank computation by further considering 
the properties of a stochastic matrix, especially the spectral properties of a stochastic or Google matrix.

\section*{References}
\bibliography{submit2arxiv}

\end{document}